\documentclass{kms-c}

\newcommand{\publname}{To appear in \enquote{ Commun.~Korean Math.~Soc.}}
\issueinfo{}% volume number
  {}%        % issue number
  {}%        % month
  {}%     % year
\pagespan{1}{}
\copyrightinfo{}%              % copyright year
  {Korean Mathematical Society}% copyright holder
\usepackage{graphicx}
\usepackage{csquotes}
\allowdisplaybreaks

\theoremstyle{plain}

\theoremstyle{definition}

\theoremstyle{remark}

\newtheorem*{theoA}{Theorem A}
\newtheorem*{theoB}{Theorem B}
\newtheorem*{theoC}{Theorem C}
\newtheorem*{theoD}{Theorem D}
\newtheorem*{theoE}{Theorem E}
\newtheorem*{theoF}{Theorem F}
\newtheorem*{theoG}{Theorem G}

\newtheorem{theo}{Theorem}[section]
\newtheorem{lem}{Lemma}[section]

\newtheorem{exm}{Example}[section]
\newtheorem{defi}{Definition}[section]
\newtheorem{rem}{Remark}[section]
\newtheorem{question}{Question}[section]
\newcommand{\ol}{\overline}
\newcommand{\be}{\begin{equation}}
\newcommand{\ee}{\end{equation}}
\newcommand{\beas}{\begin{eqnarray*}}
\newcommand{\eeas}{\end{eqnarray*}}
\newcommand{\bea}{\begin{eqnarray}}
\newcommand{\eea}{\end{eqnarray}}

\begin{document}

\title[A note on the value distribution of Differential Polynomials]
{A note on the value distribution of Differential Polynomials}

\author[S. S. Bhoosnurmath]{Subhas S. Bhoosnurmath}
\address{Subhas S. Bhoosnurmath \\ Department of Mathematics\\  Karnatak University\\  Dharwad 580 003\\ Karnatak, India}
\email{ssbmath@gmail.com}

\author[B. Chakraborty]{Bikash Chakraborty}
\address{Bikash Chakraborty \\ Department of Mathematics\\ Ramakrishna Mission Vivekananda Centenary College\\ Kolkata 700118\\West Bengal, India}
\email{bikashchakraborty.math@yahoo.com, bikash@rkmvccrahara.org}

\author[H. M. Srivastava]{H. M. Srivastava}
\address{H. M. Srivastava \\Department of Mathematics and Statistics\\ University of Victoria\\ Victoria, British Columbia V8W 3R4, Canada}
\address{H. M. Srivastava \\Department of Medical Research\\China Medical University Hospital\\ China Medical University\\ Taichung 40402\\ Taiwan, Republic of China}
\email{harimsri@math.uvic.ca}
%\thanks{This work was financially supported by KRF 2003-041-C20009}


\begin{abstract}
Let $f$ be a transcendental meromorphic function, defined in the complex plane $\mathbb{C}$. In this paper, we give a quantitative estimations of the characteristic function $T(r,f)$ in terms of the counting function of a homogeneous differential polynomial generated by $f$. Our result improves and generalizes some recent results.
\end{abstract}

\maketitle

\section{Introduction}
Throughout of this article, we adopt the standard notations and results of classical value distribution theory [See, Hayman's Monograph (\cite{8})]. A meromorphic function $g$ is said to be rational if and only if $T(r,g)=O(\log r)$, otherwise, $g$ is called transcendental meromorphic function. Let $f$ be a transcendental meromorphic function, defined in the complex plane $\mathbb{C}$.\par
We denote by $S(r, f)$, the quantity satisfying $$S(r, f) = o(T(r, f))~~\text{as}~~r\to\infty,~ r\not\in E,$$ where $E$ is a subset of positive real numbers of finite linear measure, not necessarily the same at each occurrence.\par
In addition, in this paper, we also use another type of notation $S^{*}(r,f)$ which is defined as
$$S^{*}(r,f)=o(T(r,f))~~\text{as}~~r\to\infty,~r\not\in E^{*},$$
where $E^{*}$ is a subset of positive real numbers of logarithmic density $0$.
\begin{defi}
A meromorphic function $b(z)(\not\equiv 0,\infty)$ defined in $\mathbb{C}$ is called a \enquote{small function} with respect to $f$ if $T(r,b(z))=S(r,f)$.
\end{defi}
\begin{defi}
Let $k$ be a positive integer, for any constant $a$ in the complex plane. We denote
\begin{enumerate}
\item [i)] by $N_{k)}(r,\frac{1}{( f -a)})$ the counting function of $a$-points of $f$ with multiplicity $\leq k$,
\item [ii)] by $N_{(k}(r,\frac{1}{( f -a)})$ the counting function of $a$-points of $f$ with multiplicity $\geq k$.
\end{enumerate}
Similarly, the reduced counting functions $\ol{N}_{k)}(r,\frac{1}{( f -a)})$ and $\ol{N}_{(k}(r,\frac{1}{( f -a)})$ are defined.
\end{defi}
\begin{defi}(\cite{ld})
For a positive integer $k$, we denote $N_{k}(r,0;f)$ the counting function of zeros of $f$, where a zero of $f$ with multiplicity $q$ is counted $q$ times if $q\leq k$, and is counted $k$ times if  $q> k$.
\end{defi}
In 1979, E. Mues (\cite{m}) proved that for a transcendental meromorphic function $f(z)$ in $\mathbb{C}$, $f^{2}f'-1$ has infinitely many zeros.\par
Later, in 1992, Q. Zhang (\cite{qz}) proved a quantitative version of Mues's result as follows:
\begin{theoA} For a transcendental meromorphic function $f$, the following inequality holds :
$$T(r,f)\leq 6N\bigg(r,\frac{1}{f^{2}f'-1}\bigg)+S(r,f).$$
\end{theoA}
In this direction, X. Huang and Y. Gu, (\cite{hg}) further generalised Theorem A. They proved the following result:
\begin{theoB} Let $f$ be a transcendental meromorphic function and $k$ be a positive integer. Then
$$T(r,f)\leq 6N\bigg(r,\frac{1}{f^{2}f^{(k)}-1}\bigg)+S(r,f).$$
\end{theoB}
Next we recall the following definition:
\begin{defi} Let $q_{0j},q_{1j},\ldots,q_{kj}$ be non negative integers. Then the expression
$$M_{j}[f]=(f)^{q_{0j}}(f^{(1)})^{q_{1j}}\ldots(f^{(k)})^{q_{kj}}$$
is called a differential monomial generated by $f$ of degree $d(M_{j})=\sum\limits_{i=0}^{k}q_{ij}$ and weight
$\Gamma_{M_{j}}=\sum\limits_{i=0}^{k}(i+1)q_{ij}$. The sum
$$P[f]=\sum\limits_{j=1}^{t}b_{j}M_{j}[f]$$ is called a differential polynomial generated by $f$ of degree $\ol{d}(P)=\max\{d(M_{j}):1\leq j\leq t\}$
and weight $\Gamma_{P}=\max\{\Gamma_{M_{j}}:1\leq j\leq t\}$, where $T(r,b_{j})=S(r,f)$ for $j=1,2,\ldots,t$.\par
The numbers $\underline{d}(P)=\min\{d(M_{j}):1\leq j\leq t\}$ and $k$(the highest order of the derivative of $f$ in $P[f]$) are called respectively the lower degree and order of $P[f]$.\par
The differential polynomial $P[f]$ is said to be homogeneous if $\ol{d}(P)$=$\underline{d}(P)$, otherwise, $P[f]$ is called non-homogeneous Differential Polynomial. The term $d(P)=\ol{d}(P)=\underline{d}(P)$ is called the degree of the homogeneous differential polynomial $P[f]$.\par
We also denote by $\nu =\max\; \{\Gamma _{M_{j}}-d(M_{j}): 1\leq j\leq t\}=\max\; \{ q_{1j}+2q_{2j}+\ldots+kq_{kj}: 1\leq j\leq t\}$.
\end{defi}
In 2003, I. Lahiri and S. Dewan (\cite{ld}) considered the value distribution of a differential polynomial in more general settings.
They proved the following theorem.
\begin{theoC}
Let $f$ be a transcendental meromorphic function and $\alpha=\alpha(z)(\not\equiv 0,\infty)$ be a small function of $f$. If  $\psi=\alpha(f)^{n}(f^{(k)})^{p}$, where $n(\geq 0)$ $p(\geq 1)$, $k(\geq 1)$ are integers, then for any small function $a=a(z)(\not\equiv 0,\infty)$ of $\psi$,
$$(p + n)T(r, f)\leq \overline{N}(r,\infty; f) + \overline{N}(r, 0; f) + pN_{k}(r, 0; f) + \overline{N}(r, a; \psi) + S(r, f).$$
\end{theoC}
The following result is an immediate corollary of the above theorem.
\begin{theoD} Let $f$ be a transcendental meromorphic function. Let $l(\geq3)$, $n(\geq1)$, $k(\geq1)$ be positive integers. Then
$$T(r,f)\leq \frac{1}{l-2}\ol{N}\bigg(r,\frac{1}{f^{l}(f^{(k)})^{n}-1}\bigg)+S(r,f).$$
\end{theoD}

Since $f^{l}(f^{(k)})$ is a specific form of a differential monomial, so the following questions are natural:
\begin{question}(\cite{ch}) Does there exist  positive constants $B_{1}(>0)$, $B_{2}(>0)$ such that the following two inequalities hold?
\begin{enumerate}
\item [i)] $T(r,f)\leq B_{1}~ N\bigg(r,\frac{1}{M[f]-c}\bigg)+S(r,f),$
\item [ii)] $T(r,f)\leq B_{2}~\ol{N}\bigg(r,\frac{1}{M[f]-c}\bigg)+S(r,f),$
\end{enumerate}
where $M[f]$ is any differential monomial as defined above generated by a transcendental meromorphic function $f$ and $c$ is any non zero constant.
\end{question}
In connection to the above questions, the following theorems are given in (\cite{ch}).
\begin{theoE}(\cite{ch}) Let $f$ be a transcendental meromorphic function defined in $\mathbb{C}$, and $M[f]=a (f)^{q_{0}}(f')^{q_{1}}...(f^{(k)})^{q_{k}}$ be a differential monomial generated by $f$, where $a$ be a non zero complex constant; and $k(\geq2)$, $q_{0}(\geq2)$, $q_{i}(\geq0)~(i=1,2,..,k-1)$, $q_{k}(\geq2)$ be integers. Let $\mu=q_{0}+q_{1}+...+q_{k}$ and $\mu_{*}=q_{1}+2q_{2}+...+kq_{k}$. Then
\bea\label{eq0.1} T(r,f)\leq \frac{1}{q_{0}-1}N\bigg(r,\frac{1}{M[f]-1}\bigg)+S^{*}(r,f),\eea
where $S^{*}(r,f)=o(T(r,f))$ as $r\to\infty,r\not\in E$, $E$ is a set of logarithmic density $0$.
\end{theoE}
Thus Theorem E improves, extends and generalizes Theorem D.
\begin{theoF}(\cite{ch}) Let $f$ be a transcendental meromorphic function defined in $\mathbb{C}$, and $M[f]=a (f)^{q_{0}}(f')^{q_{1}}...(f^{(k)})^{q_{k}}$ be a differential monomial generated by $f$, where $a$ be a non zero complex constant; and $k(\geq1)$, $q_{0}(\geq1)$, $q_{i}(\geq0)~(i=1,2,..,k-1)$, $q_{k}(\geq1)$ be integers. Let $\mu=q_{0}+q_{1}+...+q_{k}$ and $\mu_{*}=q_{1}+2q_{2}+...+kq_{k}$. If $\mu-\mu_{*}\geq3$, then
\bea\label{eq0.3} T(r,f)\leq \frac{1}{\mu-\mu^{*}-2}\ol{N}\bigg(r,\frac{1}{M[f]-1}\bigg)+S(r,f),\eea
where $S(r,f)=o(T(r,f))$ as $r\to\infty,r\not\in E$, $E$ is a set of finite linear measure.
\end{theoF}
\begin{theoG}(\cite{ch}) Let $f$ be a transcendental meromorphic function defined in $\mathbb{C}$, and $M[f]=a (f)^{q_{0}}(f')^{q_{1}}...(f^{(k)})^{q_{k}}$ be a differential monomial generated by $f$, where $a$ be a non zero complex constant; and $k(\geq1)$, $q_{0}(\geq1)$, $q_{i}(\geq0)~(i=1,2,..,k-1)$, $q_{k}(\geq1)$ be integers. Let $\mu=q_{0}+q_{1}+...+q_{k}$ and $\mu_{*}=q_{1}+2q_{2}+...+kq_{k}$. If $\mu-\mu_{*}\geq 5-q_{0}$, then
\bea\label{eq0.5} T(r,f)\leq \frac{1}{\mu-\mu^{*}-4+q_{0}}\ol{N}\bigg(r,\frac{1}{M[f]-1}\bigg)+S(r,f).\eea
\end{theoG}
The aim of this paper is to extend \enquote{Theorem D - Theorem G} for homogeneous differential polynomials.
%%%%%%%%%%%%%%%%%%%%%%%%%%%%%%%%%%%%%%%%%%%%%%%%%%%%%%%%%%%%%%%%%%%%%%%%%%%%%%%%%%%%%%%%%%%%%%%%%%%%%%%%%%%%%%%%%%%%%%%%%%%%%%%%%%%%%%%%%%%%%%%%%%%%%%%%%%
\section{Main Results}
\begin{theo}\label{th1} Let $f$ be a transcendental meromorphic function and $P[f]$ be a homogeneous differential polynomial generated by $f$. Let $k(\geq2)$ be the highest order of the derivative of $f$ in $P[f]$. If $q_{0j}(\geq2)$, $q_{ij}(\geq0)$ $(i=1,2,\ldots,k-1)$, $q_{kj}(\geq2)$ be integers for $1\leq j \leq t$, then either
\begin{enumerate}
\item [i)] $P[f]\equiv0$, or
\item [ii)] $T(r,f) \leq \frac{1}{q_{0}-1}N\left(r,\frac{1}{P[f]-1}\right)+S^{*}(r,f),$
\end{enumerate}
where $S^{*}(r,f)=o(T(r,f))$ as $r\to\infty$ and $r\not\in E$, $E$ is a set of logarithmic density $0$.
\end{theo}
\begin{exm} Let us take $f(z)=e^{z}$ and $$P[f]=f^{2}\left(f'(f'')^{2}(f^{(3)})^{2}-(f')^{2}f''(f^{(3)})^{2}\right).$$ Then $P[f]\equiv 0$. Thus the condition (i) in Theorem \ref{th1} is necessary.
\end{exm}
%\begin{rem} Thus the Theorem \ref{th1} extends and generalizes Theorem D.
%\end{rem}
\begin{rem} If $M[f]$ is differential monomial and $f$ is a transcendental meromorphic function, then $M[f]\not\equiv 0$.  Thus for the case of a differential monomial, the condition (i) in Theorem \ref{th1} is redundant.
\end{rem}
%\begin{rem} Under the suppositions of Theorem \ref{th1}, it is clear that $P[f]$ is either identically zero or transcendental.
%\end{rem}
\begin{rem} Under the suppositions of Theorem \ref{th1}, it is clear that either $P[f]$ is identically zero or the equation $P[f]=1$ has infinitely many solutions.
\end{rem}
\begin{theo}\label{th2} Let $f$ be a transcendental meromorphic function. Also, let $P[f]$ be a homogeneous differential polynomial generated by $f$, and of degree $d(P)$. Let $k(\geq 1)$ be the highest order of the derivative of $f$ in $P[f]$ and $q_{0j}(\geq1)$, $q_{ij}(\geq0)$ $(i=1,2,\ldots,k-1)$, $q_{kj}(\geq1)$ be integers for $1\leq j\leq t$. If $d(P)-\nu>2$, then either
\begin{enumerate}
\item [i)]  $P[f]\equiv 0$ or,
\item [ii)] $T(r,f) \leq \frac{1}{d(P)-\nu-2}\ol{N}(r,\frac{1}{P[f]-1})+S(r,f),$
\end{enumerate}
where $S(r,f)=o(T(r,f))$ as $r\to\infty$ and $r\not\in E$, $E$ is a set of finite linear measure.
\end{theo}
\begin{exm} Let us take $f(z)=e^{-z}$ and $$P[f]=f^{6}\left(f'f^{(3)}+f''f^{(3)}\right).$$ Then $P[f]\equiv 0$. Thus the condition (i) in Theorem \ref{th2} is necessary.
\end{exm}
%\begin{rem} Thus the Theorem \ref{th2} extends and generalizes Theorem E.
%\end{rem}
\begin{rem} If $M[f]$ is differential monomial and $f$ is a transcendental meromorphic function, then $M[f]\not\equiv 0$.  Thus for the case of a differential monomial, the condition (i) in Theorem \ref{th2} is redundant.
\end{rem}
%\begin{rem} Under the suppositions of Theorem \ref{th2}, it is clear that $P[f]$ is either identically zero or transcendental.
%\end{rem}
\begin{rem} Under the suppositions of Theorem \ref{th2}, it is clear that either $P[f]$ is identically zero or the equation $P[f]=1$ has infinitely many solutions.
\end{rem}
\begin{theo}\label{th3} Let $f$ be a transcendental meromorphic function. Also, let $P[f]$ be a homogeneous differential polynomial generated by $f$ and of degree $d(P)$. Let $k(\geq 1)$ be the highest order of the derivative of $f$ in $P[f]$ and $q_{0j}(\geq1)$, $q_{ij}(\geq0)$ $(i=1,2,\ldots,k-1)$, $q_{kj}(\geq1)$ be integers for $1\leq j\leq t$. If $d(P)+k q_{\ast}>2(k+1)+\nu$, then either
\begin{enumerate}
\item [i)]  $P[f]\equiv 0$ or,
\item [ii)] $T(r,f) \leq \frac{(k+1)}{d(P)+k q_{\ast}-\nu-2(k+1)}\ol{N}(r,\frac{1}{P[f]-1})+S(r,f),$
\end{enumerate}
where $S(r,f)=o(T(r,f))$ as $r\to\infty$ and $r\not\in E$, $E$ is a set of finite linear measure and $q_{\ast}=\min\limits_{j=1}^{t}\{q_{0j}\}$.
\end{theo}
\begin{exm} Let us take $f(z)=e^{z}$ and $$P[f]=f^{5}(f')^{3}-f^{3}(f')^{5}.$$ Then $P[f]\equiv 0$. Thus the condition (i) in Theorem \ref{th3} is necessary.
\end{exm}
%\begin{rem} Thus the Theorem \ref{th3} extends and generalizes Theorem F.
%\end{rem}
\begin{rem} If $M[f]$ is differential monomial and $f$ is a transcendental meromorphic function, then $M[f]\not\equiv 0$.  Thus for the case of a differential monomial, the condition (i) in Theorem \ref{th3} is redundant.
\end{rem}
%\begin{rem} Under the suppositions of Theorem \ref{th3}, it is clear that $P[f]$ is either identically zero or transcendental.
%\end{rem}
\begin{rem} Under the suppositions of Theorem \ref{th3}, it is clear that either $P[f]$ is identically zero or the equation $P[f]=1$ has infinitely many solutions.
\end{rem}
%%%%%%%%%%%%%%%%%%%%%%%%%%%%%%%%%%%%%%%%%%%%%%%%%%%%%%%%%%%%%%%%%%%%%%%%%%%%%%%%%%%%%%%%%%%%%%%%%%%%%%%%%%%%%%%%%%%%%%%%%%%%
\section{Lemmas}
\begin{lem}\label{lem1} (\cite{jh}) For a non constant meromorphic function $g$, the following equality holds:
$$N\left(r,\frac{g'}{g}\right)-N\left(r,\frac{g}{g'}\right)=\ol{N}(r,g)+N\left(r,\frac{1}{g}\right)-N\left(r,\frac{1}{g'}\right).$$
\end{lem}
The next lemma plays the major role to prove Theorem \ref{th1}, which is a immediate corollary of Yamanoi's Celebrated Theorem(\cite{yam}). Yamanoi's Theorem is a correspondent result to the famous Gol'dberg Conjecture.
\begin{lem}(\cite{yam})\label{lem1.1} Let $f$ be a transcendental meromorphic function in $\mathbb{C}$ and let $k(\geq2)$ be an integer. Then
$$(k-1)\ol{N}(r,f)\leq N\left(r,\frac{1}{f^{(k)}}\right)+S^{*}(r,f),$$
where $S^{*}(r,f)=o(T(r,f))$ as $r\to\infty$ and $r\not\in E$, $E$ is a set of logarithmic density $0$.
\end{lem}
\begin{lem}\label{lem2} Let $f$ be a transcendental meromorphic function and $b(z)(\not\equiv0,\infty)$ be a small function  with respect to $f$. If $P[f]$ be a homogeneous differential polynomial of degree $d(P)$ and $q_{0j}(\geq1)$, then either $P[f]\equiv 0$ or, $b(z)P[f]$ can not be a non zero constant.
\end{lem}
\begin{proof} Let us assume that \bea\label{kabir1} b(z)P[f]\equiv C,\eea for some constant $C$.\\
If $C=0$, then there is nothing to prove. Thus we assume that  $C\not=0$. Then from equation (\ref{kabir1}) and Lemma of logarithmic derivative, we have \bea\label{kabir2} d(P)m(r,\frac{1}{f})=m\left(r,\frac{b(z)P[f]}{Cf^{d(P)}}\right)=S(r,f).\eea
Since $q_{0j}\geq 1$, it is clear from equation (\ref{kabir1}) that \bea\label{kabir3} N(r,0;f)\leq N(r,0;P[f])+S(r,f)=S(r,f).\eea
Thus $T(r,f)=S(r,f)$, which is absurd since $f$ is transcendental meromorphic function.
\end{proof}
\begin{lem}\label{lem3}
Let $f$ be a transcendental meromorphic function, then $$T\bigg(r,b(z)P[f]\bigg)=O(T(r,f))~~\text{and}~~S\bigg(r,b(z)P[f]\bigg)=S(r,f).$$
\end{lem}
\begin{proof} The proof is similar to the proof of the Lemma 2.4 of (\cite{ly}).
\end{proof}
%%%%%%%%%%%%%%%%%%%%%%%%%%%%%%%%%%%%%%%%%%%%%%%%%%%%%%%%%%%%%%%%%%%%%%%%%%%%%%%%%%%%%%%%%%%%%%%%%%%%%%%%%%%%%%%%%%%%%%%%%%%%%%%%%%%%%%%%%%%%%%%%%%
\begin{lem}\label{lem4} Let $f$ be a transcendental meromorphic function and $b(z)(\not\equiv0,\infty)$ be a small function  with respect to $f$. If $P[f]$ be a homogeneous differential polynomial of degree $d(P)$ and $q_{0j}(\geq1)$, then either $P[f]\equiv 0,$ or
\bea\label{p1}\nonumber d(P) T(r,f)&\leq& d(P) N\left(r,\frac{1}{f}\right)+ \ol{N}\left(r,\infty;f\right)+N\left(r,\frac{1}{b(z)P[f]-1}\right)\\
\nonumber &&-N\left(r,\frac{1}{(b(z)P[f])'}\right)+S(r,f).\eea
\end{lem}
\begin{proof} Assume that $P[f]\not\equiv 0$. Then by Lemma \ref{lem2}, we have $b(z)P[f]$ can
not be a non zero constant. Thus we can write
 $$\frac{1}{f^{d(P)}}=\frac{b(z) P[f]}{f^{d(P)}}-\frac{(b(z) P[f])'}{b(z) P[f]}\frac{b(z) P[f]}{f^{d(P)}}\frac{(b(z) P[f]-1)}{(b(z) P[f])'}.$$
Thus in view of first fundamental theorem, Lemma \ref{lem3} and Lemma \ref{lem1}, we have
\bea\label{kabir5} && d(P) m(r,\frac{1}{f})\\
\nonumber &\leq& m\left(r,\frac{b(z)P[f]-1}{(b(z)P[f])'}\right)+S(r,f)\\
\nonumber&\leq& T\left(r,\frac{b(z)P[f]-1}{(b(z)P[f])'}\right)-N\left(r,\frac{b(z)P[f]-1}{(b(z)P[f])'}\right)+S(r,f)\\
\nonumber&\leq& N\left(r,\frac{(b(z)P[f])'}{b(z)P[f]-1}\right)-N\left(r,\frac{b(z)P[f]-1}{(b(z)P[f])'}\right)+S(r,f)\\
\nonumber&\leq& \ol{N}(r,\infty;f)+N\left(r,\frac{1}{b(z)P[f]-1}\right)-N\left(r,\frac{1}{(b(z)P[f])'}\right)+S(r,f)\eea
Thus \beas d(P) T(r,f) &\leq& d(P) N\left(r,\frac{1}{f}\right)+ \ol{N}(r,\infty;f)+N\left(r,\frac{1}{b(z)P[f]-1}\right)\\
&&-N\left(r,\frac{1}{(b(z)P[f])'}\right)+S(r,f).\eeas
\end{proof}
%%%%%%%%%%%%%%%%%%%%%%%%%%%%%%%%%%%%%%%%%%%%%%%%%%%%%%%%%%%%%%%%%%%%%%%%%%%%%%%%%%%%%%%%%%%%%%%%%%%%%%%%%%%%%%%%%%%%%%%%%%%%%%%%%%%%%%%%%%%%%%%%%%%%%%%%%
\begin{lem}\label{lem5} Let $f$ be a transcendental meromorphic function and $b(z)(\not\equiv0,\infty)$ be a small function  with respect to $f$. If $P[f]$ be a homogeneous differential polynomial of degree $d(P)$ and $q_{0j}(\geq1)$, $q_{kj}(\geq1)$, then either $P[f]\equiv 0,$ or
\bea\label{kabir4}&& d(P) T(r,f)\\
\nonumber&\leq&\ol{N}(r,\infty;f)+\ol{N}(r,0;f)+\nu\ol{N}_{(k+1}\left(r,\frac{1}{f}\right)+\left(d(P)-q_{\ast}\right)N_{k)}\left(r,\frac{1}{f}\right)\\
\nonumber&&+\ol{N}\left(r,\frac{1}{P[f]-1}\right)-N_{0}\left(r,\frac{1}{(P[f])'}\right)+S(r,f).\eea
where $N_{0}\left(r,\frac{1}{(P[f])'}\right)$ is the counting function of the zeros of $(P[f])'$ but not the zeros of $f(P[f]-1)$ and  $q_{\ast}=\min\limits_{j=1}^{t}\{q_{0j}\}$.
\end{lem}
\begin{proof} Clearly \bea\label{eq1} && d(P)N\left(r,\frac{1}{f}\right)+N\left(r,\frac{1}{P[f]-1}\right)-N\left(r,\frac{1}{(P[f])'}\right)\\
\nonumber&=& d(P)N\left(r,\frac{1}{f}\right)+\ol{N}(r,1;P[f])+\sum\limits_{t=2}^{\infty}\ol{N}_{(t}(r,1;P[f])- N(r,0;(P[f])')\\
\nonumber&\leq& d(P) N\left(r,\frac{1}{f}\right)-N_{\star}\left(r,\frac{1}{(P[f])'}\right)+\ol{N}\left(r,\frac{1}{P[f]-1}\right)-N_{0}\left(r,\frac{1}{(P[f])'}\right),
\eea
where $N_{\star}\left(r,\frac{1}{(P[f])'}\right)$ is the counting function of the zeros of $(P[f])'$ which comes from the zeros of $f$.\par
Let $z_{0}$ be a zero of $f$ with multiplicity $q$ such that $b_{j}(z_{0})\not=0,\infty$.\\
\textbf{Case-1} If $q\leq k$, then $z_{0}$ is the zero of $(P[f])'$ of order atleast $qq_{\ast}-1$.\\
\textbf{Case-2} If $q\geq k+1$, then $z_{0}$ is the zero of $(P[f])'$ of order atleast
\beas &&\min\limits_{1\leq j\leq t}\{q_{0j}q+q_{1j}(q-1)+\ldots+q_{kj}(q-k)\}-1\\
&=&q d(P)-\max\limits_{1\leq j\leq t}\{q_{1j}+2q_{2j}+\ldots+kq_{kj}\}-1\\
&=&q d(P)-\nu-1,\eeas
where
\beas q d(P)-\nu-1&\geq& (k d(P)-\nu)+(d(P)-1)>0.\eeas
Thus \bea\label{eq2} d(P) N\left(r,\frac{1}{f}\right)&\leq& N_{\star}\left(r,\frac{1}{(P[f])'}\right)+(\nu+1)\ol{N}_{(k+1}\left(r,\frac{1}{f}\right)\\
\nonumber &&+ (d(P)-q_{\ast}) N_{k)}\left(r,\frac{1}{f}\right)+\ol{N}_{k)}\left(r,\frac{1}{f}\right)+S(r,f).\eea
Assume that $P[f]\not\equiv0$. Now using Lemma \ref{lem4} and the inequalities (\ref{eq1}),(\ref{eq2}), we have
\beas&& d(P) T(r,f)\\
&\leq& d(P) N\left(r,\frac{1}{f}\right)+ \ol{N}(r,\infty;f)+N\left(r,\frac{1}{b(z)P[f]-1}\right)\\
&&-N\left(r,\frac{1}{(b(z)P[f])'}\right)+S(r,f)\\
&\leq& N_{\star}\left(r,\frac{1}{(P[f])'}\right)+(\nu+1)\ol{N}_{(k+1}\left(r,\frac{1}{f}\right)+ (d(P)-q_{\ast}) N_{k)}\left(r,\frac{1}{f}\right)\\
&&+\ol{N}_{k)}\left(r,\frac{1}{f}\right)+ \ol{N}(r,\infty;f)+N\left(r,\frac{1}{b(z)P[f]-1}\right)\\
&&-N\left(r,\frac{1}{(b(z)P[f])'}\right)+S(r,f)\\
&\leq&\ol{N}(r,\infty;f)+\ol{N}(r,0;f)+\nu\ol{N}_{(k+1}\left(r,\frac{1}{f}\right)+ (d(P)-q_{\ast}) N_{k)}\left(r,\frac{1}{f}\right)\\
&&+\ol{N}\left(r,\frac{1}{P[f]-1}\right)-N_{0}\left(r,\frac{1}{(P[f])'}\right)+S(r,f).\eeas
\end{proof}
%%%%%%%%%%%%%%%%%%%%%%%%%%%%%%%%%%%%%%%%%%%%%%%%%%%%%%%%%%%%%%%%%%%%%%%%%%%%%%%%%%%%%%%%%%%%%%%%%%%%%%%%%%%%%%%%%%%%%%%%%%%%%%%%%%%
\section {Proof of the Theorems}
%%%%%%%%%%%%%%%%%%%%%%%%%%%%%%%%%%%%%%%%%%%%%%%%%%%%%%%%%%%%%%%%%%%%%%%%%%%%%%%%%%%%%%%%%%%%%%%%%%%%%%%%%%%%%%%%%%%%%%%%%%%%%%%%%%%
\begin{proof} [\textbf{Proof of Theorem \ref{th1}}] Given that $f$ is a transcendental meromorphic function and  $k\geq2$, $q_{0j}\geq2$, $q_{kj}\geq2$ for $1\leq j\leq t$. Let $P[f]\not\equiv 0$. Now
\bea\label{bcs}(q_{0}-1)N(r,0;f)+(q_{k}-1)N(r,0;f^{(k)})\leq N(r,0;(P[f])').\eea
Now in view of equation (\ref{bcs}) and Lemmas \ref{lem1.1} and \ref{lem4}, we have
\beas && d(P) T(r,f)\\
&\leq& d(P) N\left(r,\frac{1}{f}\right)+ \ol{N}(r,\infty;f)+N\left(r,\frac{1}{b(z)P[f]-1}\right)\\
&&-N\left(r,\frac{1}{(b(z)P[f])'}\right)+S(r,f)\\
&\leq& \left(d(P)-(q_{0}-1)\right) N\left(r,\frac{1}{f}\right)+\ol{N}(r,\infty;f)+N\left(r,\frac{1}{b(z)P[f]-1}\right)\\
&&-(q_{k}-1)N(r,0;f^{(k)})+S(r,f)\\
&\leq& \left(d(P)-(q_{0}-1)\right) N\left(r,\frac{1}{f}\right)+\left(1-(k-1)(q_{k}-1)\right)\ol{N}(r,\infty;f)\\
&&+N\left(r,\frac{1}{b(z)P[f]-1}\right)+S^{*}(r,f)+S(r,f)\\
&\leq& (d(P)-q_{0}+1)) T(r,f)+N\left(r,\frac{1}{b(z)P[f]-1}\right)+S^{*}(r,f).\eeas
i.e.,
\beas T(r,f) \leq \frac{1}{q_{0}-1}N\left(r,\frac{1}{P[f]-1}\right)+S^{*}(r,f).\eeas
This completes the proof.
\end{proof}
\begin{proof} [\textbf{Proof of Theorem \ref{th2}}] If $P[f]\not\equiv 0$, then using  Lemma \ref{lem5}, we have
\bea\label{eq5.11}\nonumber && d(P) T(r,f)\\
\nonumber&\leq&\ol{N}(r,\infty;f)+\ol{N}(r,0;f)+\nu\ol{N}_{(k+1}\left(r,\frac{1}{f}\right)+ (d(P)-q_{\ast}) N_{k)}\left(r,\frac{1}{f}\right)\\
\nonumber&&+\ol{N}\left(r,\frac{1}{P[f]-1}\right)-N_{0}\left(r,\frac{1}{(P[f])'}\right)+S(r,f)\eea
That is, 
\bea\label{eq5.1} d(P) T(r,f)&\leq& 2T(r,f)+(\nu-d(P)+q_{\ast})\ol{N}_{(k+1}\left(r,\frac{1}{f}\right)\\
\nonumber&& +(d(P)-q_{\ast}) \left(\ol{N}_{(k+1}\left(r,\frac{1}{f}\right)+N_{k)}\left(r,\frac{1}{f}\right)\right)\\
\nonumber&&+\ol{N}\left(r,\frac{1}{P[f]-1}\right)-N_{0}\left(r,\frac{1}{(P[f])'}\right)+S(r,f)\\
\nonumber&\leq& (\nu+2)T(r,f)+\ol{N}\left(r,\frac{1}{P[f]-1}\right)+S(r,f).
\eea
i.e.,
\beas T(r,f) \leq \frac{1}{d(P)-\nu-2}\ol{N}\left(r,\frac{1}{P[f]-1}\right)+S(r,f).\eeas
This completes the proof.
\end{proof}
\begin{proof} [\textbf{Proof of Theorem \ref{th3} }] If $P[f]\not\equiv 0$, then from the inequality (\ref{eq5.1}), we have
\bea\label{eq5}\nonumber && (q_{\ast}-2) T(r,f)\\
\nonumber&\leq& (\nu-d(P)+q_{\ast})\ol{N}_{(k+1}\left(r,\frac{1}{f}\right)+\ol{N}\left(r,\frac{1}{P[f]-1}\right)\\
\nonumber&& -N_{0}\left(r,\frac{1}{(P[f])'}\right)+S(r,f)\\
\nonumber&\leq& \frac{\nu-d(P)+q_{\ast}}{k+1}T(r,f)+\ol{N}\left(r,\frac{1}{P[f]-1}\right)+S(r,f).
\eea
i.e.,
\beas T(r,f) \leq \frac{(k+1)}{d(P)+k q_{\ast}-\nu-2(k+1)}\ol{N}\left(r,\frac{1}{P[f]-1}\right)+S(r,f).\eeas
This completes the proof.
\end{proof}
%%%%%%%%%%%%%%%%%%%%%%%%%%%%%%%%%%%%%%%%%%%%%%%%%%%%%%%%%%%%%%%%%%%%%%%%%%%%%%%%%%%%%%%%%%%%%%%%%%%%%%%%%%%%%%%%%%%%%%%%%%%%%%%%%%%
%%%%%%%%%%%%%%%%%%%%%%%%%%%%%%%%%%%%%%%%%%%%%%%%%%%%%%%%%%%%%%%%%%%%%%%%%%%%%%%%%%%%%%%%%%%%%%%%%%%%%%%%%%%%%%%%%%%%%%%%%%%%%%%%%%%

\end{document}